\documentclass[12pt,reqno]{amsart}
\usepackage{enumerate, latexsym, amsmath, amsfonts, amssymb, amsthm, color}
\def\pmod #1{\ ({\rm{mod}}\ #1)}
\def\Z{\mathbb Z}
\def\N{\mathbb N}
\def\Q{\mathbb Q}

\def\R{\mathbb R}
\def\l{\left}
\def\r{\right}
\def\bg{\bigg}
\def\({\bg(}
\def\){\bg)}
\def\t{\text}
\def\f{\frac}

\def\ls{\leqslant}
\def\gs{\geqslant}

\def\sm{\setminus}

\def\al{\alpha}

\def\ve{\varepsilon}

\def\eq{\equiv}

\theoremstyle{plain}
\newtheorem{theorem}{Theorem}

\newtheorem{lemma}{Lemma}

\theoremstyle{definition}

\theoremstyle{remark}
\newtheorem{remark}{Remark}

 \vspace{4mm}

\begin{document}

\hbox{Preprint}
\medskip

\title
[{Equations over the integer rings of quadratic fields}]
{On Diophantine equations over \\ the integer rings of quadratic fields}

\author
[Z.-W. Sun] {Zhi-Wei Sun}

\address{School of Mathematics, Nanjing
University, Nanjing 210093, People's Republic of China}
\email{zwsun@nju.edu.cn}

\keywords{Hilbert's Tenth Problem, Diophantine equation, quadratic field, undecidability.
\newline \indent Supported by the Natural Science Foundation of China (grant no. 12371004).}

\begin{abstract} Let $K$ be any quadratic number field, and let $O_K$ be the ring of algebraic integers in $K$. In 1975 J. Denef proved that Hilbert's Tenth Problem over $O_K$ has a negative solution.
In this paper we establish the following undecidability result: There is no algorithm to decide whether an arbitrarily given polynomial equation $P(z_1,\ldots,z_{16})=0$
(with integer coefficients and 16 unknowns) has solutions over $O_K$.
Moreover, when $K$ is a real quadratic field, we show that $15$ unknowns suffice for undecidability. 
\end{abstract}
\maketitle

\section{Introduction}
\setcounter{lemma}{0}
\setcounter{theorem}{0}
\setcounter{corollary}{0}
\setcounter{remark}{0}

The original Hilbert's Tenth Problem (HTP in short) posed in 1900 asks for an algorithm  to test whether an arbitrary polynomial Diophantine equation
$$P(x_1,\ldots,x_n)=0$$
with integer coefficients has solutions $x_1,\ldots,x_n\in\Z$. This was finally solved negatively by Y. Matiyasevich \cite{M70} in 1970 based on the earlier work \cite{DPR}. The author \cite{S21} proved that
 $\exists^{11}$ over $\Z$ is undecidable, i.e., there is no algorithm to test for any $P(x_1,\ldots,x_{11})\in\Z[x_1,\ldots,x_{11}]$
whether $$\exists x_1,\ldots,x_{11}\in\Z \,[P(x_1,\ldots,x_{11})=0].$$
See also the book \cite{S-book} for a systematic introduction to this 11 unknowns theorem and its applications.

HTP over a ring $R$ asks for an algorithm  to test whether an arbitrary polynomial Diophantine equation
$$P(x_1,\ldots,x_n)=0$$
with coefficients in $R$ has solutions over $R$.

Let $K$ be any number field
and let $O_K$ be the ring of algebraic integers in $K$.
In 1975 J. Denef \cite{D75} proved that HTP over the integer ring $O_K$
 has a negative solution if $[K:\Q]=2$.
Recently, P. Koymans and C. Pagano \cite{KP}, as well as L. Alp\"oge, M. Bhargava, W. Ho
and A. Shnidman \cite{ABHS} proved that HTP over $O_K$ is always unsolvable.

Recently, Y. Matiyasevich and the author \cite{JNT} proved that there is no algorithm to decide whether
for any $P(x_1,\ldots,x_{20})\in\Z[x_1,\ldots,x_{20}]$ the  equation
$P(x_1,\ldots,x_{20})=0$ is solvable over $\Z[i]$.  Y. Ding and J. Li \cite{DL26}
used AI to improve this via replacing $20$ by $18$. 

Let $d\not=0,1$ be a squarefree integer. For the quadratic field $K=\Q(\sqrt d)$, it is well known that
$$O_K=\begin{cases}\{\f{a+b\sqrt d}2:\ a,b\in\Z\ \t{and}\ a\eq b\pmod2\}&\t{if}\ d\eq1\pmod4,
\\\{a+b\sqrt d:\ a,b\in\Z\}&\t{otherwise}.
\end{cases}$$

 In this paper, we establish the following two theorems.
 
 \begin{theorem} \label{Th1.1} Let $d$ be any positive squarefree integer, and let $K$ be the imaginary quadratic field
 $\Q(\sqrt{-d})$. Then there is no algorithm to decide for any $P(x_1,\ldots,x_{16})\in\Z[x_1,\ldots,x_{16}]$ whether $P(x_1,\ldots,x_{16})=0$ for some $x_1,\ldots,x_{16}\in O_K$.
\end{theorem}

\begin{theorem} \label{Th1.2} Let $d>1$ be a squarefree integer, and let $K$ be the real quadratic field
 $\Q(\sqrt{d})$. Then there is no algorithm to decide for any $P(x_1,\ldots,x_{15})\in\Z[x_1,\ldots,x_{15}]$ whether $P(x_1,\ldots,x_{15})=0$ for some $x_1,\ldots,x_{15}\in O_K$.
\end{theorem}

To prove Theorems \ref{Th1.1} and \ref{Th1.2}, we need the following key result
which follows from Sun \cite[Theorem 1.1(ii)]{S21}.

\begin{theorem}[Sun \cite{S21}] \label{Th-S} There is no algorithm to decide for any $P(z_1,\ldots,z_{10})\in\Z[z_1,\ldots,z_{10}]$ whether 
$P(z_1,\ldots,z_{10})=0$ for some $z_1,\ldots,z_{10}\in\Z$ with $z_{10}\not=0$.
\end{theorem}

Theorems \ref{Th1.1} and \ref{Th1.2} will be proved in Sections 2 and 3, respectively.

\section{Proof of Theorem \ref{Th1.1}}
\setcounter{lemma}{0}
\setcounter{theorem}{0}
\setcounter{corollary}{0}
\setcounter{remark}{0}
\setcounter{equation}{0}

We need the following simple lemma (cf. \cite[Lemma 5]{DL26}) which follows from the Chinese Remainder Theorem.

\begin{lemma} \label{Lem-DL} An integer $m$ is nonzero if and only if $m\mid(2w+1)(3w+1)$
for some $w\in\Z$.
\end{lemma}
\begin{remark} S. P. Tung \cite{T85} observed that an integer $m$ is nonzero if and only if $m=(2x+1)(3y+1)$ for some $x,y\in\Z$.
\end{remark}

We also need the following general lemma.

\begin{lemma}\label{Lem-ST} Let $K$ be a number field, and let $A_1,A_2,S,T\in O_K$ with $A_1\not=A_2$
and $T\not=0$. Then
$$A_1\in\square\land A_2\in\square\land S\mid T$$
if and only if for some $m\in O_K$ we have
\begin{equation}\label{ST} F(A_1,A_2,S,T,m)=0
\end{equation}
where $\square=\{\al^2:\ \al\in O_K\}$ and
\begin{equation}\label{F}F(A_1,A_2,S,T,m)=(T-mS)^4-2(A_1+A_2)S^2(T-mS)^2+(A_1-A_2)^2S^4.
\end{equation}
\end{lemma}
\begin{proof}.
This is \cite[Lemma 3.1]{JNT} if we replace the condition $T\not=0$ by $S\not=0$.
If $S=0$ and $T\not=0$, then $S\nmid T$ and also $F(A_1,A_2,S,T,m)=T^4\not=0$ for any $m\in O_K$.
So the desired result follows.
\end{proof}

In this section, from now on, we fix a positive squarefree integer $d$, and let $K=\Q(\sqrt{-d})$
and $\square=\{\al^2:\ \al\in O_K\}$.

\begin{lemma} \label{Lem-0} For any $x,y\in O_K$, we have
\begin{equation}x=0\land y=0\iff x^2+(d+1)y^2=0.
\end{equation}
\end{lemma}
\begin{proof} If $x^2+(d+1)y^2=0$ but $x$ or $y$ is nonzero,
then $x,y\not=0$ and $-d-1=(x/y)^2$ is a square in $K=\Q(\sqrt{-d})$.

Suppose that $-d-1=(a+b\sqrt{-d})^2$ for some $a,b\in\Q$. Then $a^2-db^2=-d-1$ and $2ab=0$.
If $b=0$, then $a^2=-d-1<0$ which is impossible. Thus $a=0$ and $d(d+1)=d^2b^2$.
Hence both $d$ and $d+1$ are integer squares. As $d$ is squarefree, we must have $d=1$ and hence $d+1=2$. Since $2$ is not an integer square, we get a contradiction.

In view of the above, we have completed the proof of Lemma \ref{Lem-0}.
\end{proof}

The following lemma in the case $d=1$ is essentially due to J. Denef \cite{D75}.

\begin{lemma} \label{Lem-xy} Suppose that $x,y\in O_K$ and $x^2-3y^2=1$. Then we must have $x,y\in\Z$.
\end{lemma}
\begin{proof}
Set
$$c=\begin{cases}2&\t{if}\ d\eq3\pmod4,\\1&\t{otherwise}.\end{cases},$$
and write 
$$x=\f{a+b\sqrt{-d}}c\ \ \t{and}\ \ y=\f{s+t\sqrt{-d}}c,$$
where $a,b,s,t$ are integers such that $a\eq b\pmod2$ and $s\eq t\pmod2$ when $c=2$.
As $x^2-3y^2=1$, we have
$$(a+b\sqrt{-d})^2-3(s+t\sqrt{-d})^2=c^2,$$
i.e.,$$ (a^2-3s^2)-d(b^2-3t^2)+2(ab-3st)\sqrt{-d}=c^2.$$
It follows that
\begin{equation}\label{=}a_0-db_0=c^2\ \t{and}\ ab=3st,
\end{equation}
where $a_0=a^2-3s^2$ and $b_0=b^2-3t^2$. In view of this, we have
\begin{equation}\label{0} a_0b_0=-3(at-bs)^2\ls0.
\end{equation}

Assume that $a_0=0$. Then $a^2=3s^2$ and hence $a=s=0$. Note also that $-bb_0=c^2$.
If $c=1$, then $d=1$ and $b^2-3t^2=b_0=-1$ which is impossible since $b^2\not\eq-1\pmod3$.
If $c=2$, then $d$ divides $c^2=4$ which is impossible since $d\eq3\pmod4$.

By the last paragraph, we have $a_0\not=0$. If $a_0<0$, then $b_0\gs0$ by \eqref{0}, and hence
$c^2=a_0-db_0<0$ which is impossible. Thus $a_0>0$ and $b_0\ls0$.

Suppose that $b_0<0$. Then $0<a_0=db_0+c^2<c^2$. Hence $c=2$ and $a_0\in\{1,2,3\}$. 
Note that $a_0=a^2-3s^2\not\eq2\pmod3$. If $a_0=3$, then for some $a'\in\Z$ we have
$3=a_0=(3a')^2-3s^2$ and hence $1=3(a')^2-s^2\eq -s^2\pmod3$ which is impossible.
Therefore   $a_0=1$. Hence
$db_0=a_0-c^2=1-2^2=-3$ and $b_0=a_0b_0=-3(at-bs)^2$.
Thus $d(at-bs)^2=1$ and hence $d=1$ which contradicts $c=2$.

By the above, we must have $b^2-3t^2=b_0=0$. So $b=t=0$. It follows that
$x,y\in\Q\cup O_K=\Z$ as desired. 
\end{proof}

\begin{lemma}\label{Lem-z} Let $z\in O_K$. Then $z\in\Z$ if there are 
$v,x,y\in O_K$ with $v\not=0$ such that $f(v,y,z)=0$ and $g(x,y,z)\in\square$,
where
$$f(v,y,z)=4(2v(2(2z+1)^2+1)-y)^2-3y^2-1$$
and
$$g(x,y,z)=3y^2(2z+1-xy)^2+1.$$
\end{lemma}
\begin{proof} Suppose that
$f(v,y,z)=0$ and $g(x,y,z)\in\square$ for some $v,y,z\in O_K$ with $v\not=0$.
Set $w=2(2z+1)^2+1$ and $x=2(2vw-y)=4vw-2y$. As $f(v,y,z)=0$, we have
$x^2-3y^2=1$ and hence $x,y\in\Z$ by Lemma \ref{Lem-xy}. Note that $w\not=0$
since $-1/2\not\in O_K$.

 If $y=0$, then $4vw=x=\pm1$ which is impossible since $\pm1/4\not\in O_K$.

Assume that $|y|=1$. Then $x^2=3y^2+1=4$, and hence $x=2y$ since $x+2y=4vw\not=0$.
Thus $4vw=x+2y=4y$ and hence $vw=y\in\{\pm1\}$. So $w$ belongs to the unit group $U(O_K)$ of
the ring $O_K$. Note that $w=2(2z+1)^2+1\eq1\pmod{2}$ in the ring of all algebraic integers. 
If $d=1$, then $U(O_K)=\{pm1,\pm i\}$ and $\pm i\not\eq1\pmod{2}$. If $d=3$ and $\omega=(-1+\sqrt{-3})/2$, then $U(O_K)=\{\pm1,\pm\omega,\omega^2\}$ and
$\omega,\omega^2\not\eq1\pmod{2}.$
It is well known that $U(O_K)=\{pm1\}$ if $d\not=1,3$. 
Thus we must have $w=\pm1$. As $2z+1\ne0$, we have $w\not=1$ and hence $w=-1$.
Thus $(2z+1)^2=-1$ and $2z+1=\pm i$. Since $(-1\pm i)/2$ is not an algebraic integer,
we get a contradiction.

By the last two paragraph, we must have $|y|\gs2$. As $|v|^2=v\bar v\in\{1,2,3,\ldots\}$,
we have $|v|\gs1$. Recall that $x+2y=4vw$. So 
\begin{equation}\label{wy}|w|\ls\f{|x|+2|y|}{4|v|}\ls\f{\sqrt{3y^2+1}+2|y|}4<|y|.
\end{equation}

Write $g(x,y,z)=u^2$ with $u\in O_K$. As $u^2-3(y(2z+1-xy))^2=1$, we have $u,y(2z+1-xy)\in\Z$
by Lemma \ref{Lem-xy}. Recall that $y\in\Z\sm\{0\}$. So $2z+1-xy\in\Q\cap O_K=\Z$. 
If $x\in\Z$, then $2z+1\in \Z$ and hence $z=\f{(2z+1)-1}2\in\Q\cap O_K=\Z$. 

Now suppose that $x\not\in\Z$. Let $c$ take $2$ or $1$ according as $d\eq3\pmod4$ or not.
Observe that
\begin{align*}|2z+1|&\gs|\mathrm{Im}(2z+1)|=|\mathrm{Im}(2z+1-xy+xy)|
\\&=|\mathrm{Im}(xy)|=|\mathrm{Im}(x)|\cdot|y|
\\&\gs\f{\sqrt d}c|y|\gs\f{\sqrt3}2|y|
\end{align*}
and hence
$$|w|\gs2|2z+1|^2-1\gs 2\times\f{3y^2}4-1=\f 32y^2-1>|y|$$
which contradicts \eqref{wy}.

In view of the above, we must have $x\in\Z$ and hence $z\in\Z$ as desired.
This concludes the proof. 
\end{proof}

\begin{remark} Lemma \ref{Lem-z} in the case $d=1$ appeared in \cite[Theorem 1.1]{MS}.
\end{remark}

\begin{lemma} \label{Lem-z12} Let $z_1,z_2\in O_K$, and set $w_j=2(2z_j+1)^2+1$ for $j=0,1$.
Then $z_1,z_2\in\Z$ if and only if there are $v,x_1,x_2,y\in O_K$ with $v\not=0$ such that
\begin{equation}\label{double}
4(2vw_1w_2-y)^2-3y^2-1=0\ \ \t{and}\ \ g(x_1,y,z_1),g(x_2,y,z_2)\in\square.
\end{equation}
Moreover, when $z_1,z_2\in\Z$ we may require further that $v,x_1,x_2,y\in\Z$.
\end{lemma}
\begin{proof} We first prove the ``only if" direction. Suppose that there are $v,x_1,x_2,y\in O_K$
with $v\not=0$ such that \eqref{double} holds. Note that $w_1w_2\not=0$ since $-1/2\not\in O_K$.
As $vw_2\not=0$, $f(vw_2,y,z_1)=0$ and $g(x_1,y,z_1)\in\square$,
by Lemma \ref{Lem-z} we have $z_1\in\Z$. Similarly, since $vw_1\not=0$,
$f(vw_1,y,z_2)=0$ and $g(x_2,y,z_2)\in\square$, we have $z_2\in\square$.

Now we consider the ``if direction". 
As in \cite{MS}, we define
$$u_0=0,\ u_1=1, \ \t{and}\ u_{n+1}=4u_n-u_{n-1}\ \t{for}\ n=1,2,3,\ldots.$$ 
By \cite[Lemmas 8 and 9]{S92}, 
$$u_{n+1}\gs 3^n\ \ \t{and}\ \ u_{n+1}^2-4u_nu_{n+1}+u_n^2=1$$
for all $n\in\N$. For any $k,n\in\Z^+$, by \cite[Lemma 2]{S92} we have 
$$u_{nk}\eq k(u_{n+1}-4u_n)^{k-1}u_n\pmod{u_n^2}$$
and hence $q=u_{kn}/u_n\in\Z^+$. If $k$ is odd, then
$$q\eq ku_{n+1}^{k-1}\eq k\pmod {u_n}$$
since $u_{n+1}^2=1-u_n^2+4u_nu_{n+1}\eq 1\pmod{u_n}$.

Let $z_1,z_2\in\Z$. 
By \cite[Lemma 6]{S92}, there is a positive integer $r$ such that
$$u_r\eq 0\pmod{4w_1w_2}\ \ \t{and}\ \ u_{r+1}\eq1\pmod{4w_1w_2}.$$ 
Then $n=r-1$ and $v=u_{n+1}/(4w_1w_2)$ are positive integers. Set $y=u_n$. Observe that
$$4(2vw_1w_2-y)^2-3y^2=4\l(\f{u_{n+1}}2-u_n\r)^2-3u_n^2=u_{n+1}^2-4u_nu_{n+1}+u_n^2=1.$$

Let $j\in\{1,2\}$. Put $k_j=|2z_j+1|$ and write $2z_j+1=\ve_jk_j$ with $\ve_j\in\{\pm1\}$. Then $q_j=u_{k_jn}/u_n\in\Z$ and $q_j\eq k_j\pmod {u_n}$. Note that
$x_j=\ve_j(k_j-q_j)/u_n\in\Z$ and 
$$2z_j+1-x_jy=\ve_kk_j-\ve_j(k_j-q_j)=\ve_jq_j.$$
Therefore,
\begin{align*}g(x_j,y,z_j)&=3y^2(2z_j+1-x_jy)^2+1=3u_n^2q_j^2+1
\\&=3u_{k_jn}^2+1=(u_{k_jn+1}-2u_{k_jn})^2\in\square.
\end{align*}

In view of the above, we have finished the proof of Lemma \ref{Lem-z12}. 
\end{proof}

We also need the following lemma which is \cite[Theorem 1.2]{MS}.

\begin{lemma} \label{Lem-MS} Let $x_1,\ldots,x_n\in O_K$
and set $y=2\prod_{k=1}^n(3x_k+1)$. Then we have
$$y+\sum_{k=1}^n\f{x_k}{y^k}\in\Q\iff x_1,\ldots,x_n\in\Z.$$
\end{lemma}

\medskip
\noindent{\tt Proof of Theorem \ref{Th1.1}}. Let $z_1,\ldots,z_{10}\in O_K$. Then $y=2\prod_{k=1}^{10}(3z_k+1)\not=0$.
By Lemma \ref{Lem-MS}, 
\begin{equation}\label{Z} z_1,\ldots,z_{10}\in\Z
\iff\ \f{Z_2}{Z_1}=y+\sum_{k=1}^{10}\f{z_k}{y^k}\in\Q,
\end{equation}
where $Z_1=y^{10}$ and $Z_2=y^{11}+\sum_{k=1}^{10}z_ky^{10-k}$.
By Lemma \ref{Lem-z12}, $Z_1,Z_2\in\Z$ if and only if there are
$v,x_1,x_2,y_{0}\in O_K$ with $v\not=0$ such that
\begin{equation}\label{w12}4(2vw_1w_2-y_0)^2-3y_0^2-1=0
\end{equation}
and
\begin{equation}\label{g} g(x_1,y_0,Z_1),g(x_2,y_0,Z_2)\in\square,
\end{equation}
where $w_j=2(2Z_j+1)^2+1$ for $j=1,2$; moreover, when $Z_1,Z_2\in\Z$ we may require further that $v,x_1,x_2,y_0\in\Z$.  
When $v,x_1,x_2,y_0\in O_K$, clearly
$9g(x_1,y_0,z_1)\not=g(x_2,y_0,Z_2)$ since $g(x_j,y_0,z_j)\eq1\pmod3$, hence
for any $S,T\in O_K$ with $T\not=0$ we have
\begin{align*}&\ g(x_1,y_0,Z_1),g(x_2,y_0,Z_2)\in\square\land S\mid T
\\\iff&\ 9g(x_1,y_0,Z_1),g(x_2,y_0,Z_2)\in\square\ \land\ S\mid T
\\\iff&\ \exists m\in O_K\,[F(9g(x_1,y_0,Z_1),g(x_2,y_0,Z_2),S,T,m)=0]
\end{align*}
by Lemma \ref{ST}. 

Let $P$ be any polynomial in $10$ variables with integer coefficients.

Suppose that $P(z_1,\ldots,z_{10})=0$ for some $z_1,\ldots,z_{10}\in\Z$ with $z_{10}\not=0$. Then
$y=2\prod_{k=1}^{10}(3z_k+1)\in\Z\sm\{0\}$ and $Z_1,Z_2\in\Z$. Hence there are $v,x_1,x_2,y_0\in\Z$ with $v\not=0$ satisfying both \eqref{w12} and \eqref{g}. Since $vz_{10}\not=0$, by Lemma \ref{Lem-DL} we have $vz_{10}\mid(2w+1)(3w+1)$ for some $w\in\Z$, thus 
\begin{equation}\label{m}F(9g(x_1,y_0,Z_1),g(x_2,y_0,Z_2),vz_{10},(2w+1)(3w+1),m)=0
\end{equation}
for some $m\in\Z$. So we have
\begin{equation}\label{P16}\tilde P(z_1,\ldots,z_{10},v,x_1,x_2,y_0,w,m)=0,
\end{equation}
where $\tilde P$ denotes the expression
\begin{align*}&(P(z_1,\ldots,z_{10})^2+(d+1)(4(2vw_1w_2-y_0)^2-3y_0^2-1)^2)^2
\\&\ +(d+1)F(9g(x_1,y_0,Z_1),g(x_2,y_0,Z_2),vz_{10},(2w+1)(3w+1),m)^2.
\end{align*}

Now assume that \eqref{P16} holds for some $z_1,\ldots,z_{10},v,x_1,x_2,y_0,w,m\in O_K$.
In light of Lemma \ref{Lem-xy}, $P(z_1,\ldots,z_{10})=0$, and both \eqref{w12} and \eqref{m} hold.
Since $-1/2,-1/3\not\in O_K$, we have $(2w+1)(3w+1)\not=0$. Thus, by Lemma \ref{Lem-ST} and the equality \eqref{m}, we have \eqref{g} and $vz_{10}\mid (2w+1)(3w+1)$. As $v\not=0$ and both \eqref{w12} and \eqref{g} hold, we have $Z_1,Z_2\in\Z$ and hence  $z_1,\ldots,z_{10}\in\Z$ by \eqref{Z}.
Note that $z_{10}\not=0$ since it divides $(2w+1)(3w+1)$. 

By the last two paragraphs, $P(z_1,\ldots,z_{10})=0$ for some $z_1,\ldots,z_{10}\in\Z$ with $z_{10}\not=0$ if and only if \eqref{P16} holds for some
$$z_1,\ldots,z_{10},v,x_1,x_2,y_0,w,m\in O_K.$$ 
Thus, by applying Theorem \ref{Th-S} we obtain the desired result.
 This completes our proof of Theorem \ref{Th1.1}.
\qed

\section{Proof of Theorem \ref{Th1.2}}
\setcounter{lemma}{0}
\setcounter{theorem}{0}
\setcounter{corollary}{0}
\setcounter{remark}{0}
\setcounter{equation}{0}

Throughout this section, we fix a squarefree integer $d>1$ and  put
 $K=\Q(\sqrt d)$. We also
let $\sigma_1,\sigma_2:K\hookrightarrow\R$ be the two real embeddings. 

\begin{lemma}\label{lem:real-separation}
If $\alpha\in O_K\setminus\Z$, then
\[
 |\sigma_1(\alpha)-\sigma_2(\alpha)|\gs\sqrt d>1.
\]
\end{lemma}

\begin{proof}
If $d\not\equiv1\pmod4$, write $\alpha=a+b\sqrt d$ with $a,b\in\Z$. Since $\alpha\notin\Z$, one has $b\ne0$, and the difference of the two conjugates has absolute value $2|b|\sqrt d$. If $d\equiv1\pmod4$, write $\alpha=a+b(1+\sqrt d)/2$ with $a,b\in\Z$. Again $b\ne0$, and the difference has absolute value $|b|\sqrt d$.
\end{proof}

\begin{lemma}\label{Lem3.2}
Let $z,x_1,\ldots,x_m,c\in O_K$ with $m\gs0$ and $\sigma_i(c)\gs1$ for $i=1,2$.
Assume that
\[
 Y=2c\prod_{j=1}^m(x_j^2+1)
\]
is a positive rational integer. If
\[
 z+\sum_{j=1}^m\frac{x_j}{Y^j}\in\Q,
\]
then $z\in\Z$.
\end{lemma}

\begin{proof}
For $i=1,2$ write $x_j^{(i)}=\sigma_i(x_j)$. Since $t^2+1\gs2|t|$ for every real $t$, and every omitted factor has real value at least $1$, we have
\[
 Y\gs2\bigl((x_j^{(i)})^2+1\bigr)\gs4|x_j^{(i)}|.
\]
Therefore
\[
 \left|\sum_{j=1}^m\frac{x_j^{(i)}}{Y^j}\right|
 \ls\frac14\sum_{j=1}^mY^{1-j}<\frac12,
\]
with the assertion immediate when $m=0$. The assumed rationality then gives
$|\sigma_1(z)-\sigma_2(z)|<1$. Thus $z\in\Z$ by Lemma \ref{lem:real-separation}.
\end{proof}

Now we can give an important auxiliary theorem, which is similar to Lemma \ref{Lem-MS}
(i.e., \cite[Theorem 1.2]{JNT}). 

\begin{theorem}\label{Th3.1}
Let $x_1,\ldots,x_n\in O_K$, and define
\begin{equation}\label{eq:realY}
 Y=2\prod_{j=1}^n(x_j^2+1)
\end{equation}
and
\begin{equation}\label{eq:realW}
 W=Y^{n+1}+\sum_{j=1}^n x_jY^{n-j}.
\end{equation}
Then
\[
 W\in\Q\quad\Longleftrightarrow\quad x_1,\ldots,x_n\in\Z.
\]
\end{theorem}

\begin{proof}
The forward implication from integral $x_j$ is immediate. Conversely, suppose $W\in\Q$. Put
\[
 Y_i=\sigma_i(Y)>0,
 \qquad
 R_i=\sum_{j=1}^n\sigma_i(x_j)Y_i^{n-j}
 \qquad(i=1,2).
\]
As above, $|\sigma_i(x_j)|\ls Y_i/4$. Since $Y_i\gs2$, we have
\begin{equation}\label{eq:Rbound}
 |R_i|\ls\frac{Y_i^n}{4}\sum_{r=0}^{n-1}Y_i^{-r}<\frac{Y_i^n}{2}.
\end{equation}
We claim $Y_1=Y_2$. If $Y_1>Y_2$, then $Y\notin\Z$, and Lemma \ref{lem:real-separation} gives $Y_1-Y_2>1$. Hence
\[
 Y_1^{n+1}-Y_2^{n+1}
 =(Y_1-Y_2)\sum_{r=0}^{n}Y_1^{n-r}Y_2^r>Y_1^n.
\]
On the other hand, $\sigma_1(W)=\sigma_2(W)$ and \eqref{eq:Rbound} give
\[
 Y_1^{n+1}-Y_2^{n+1}=R_2-R_1
 <\frac{Y_1^n+Y_2^n}{2}<Y_1^n,
\]
a contradiction. The other ordering is symmetric. Hence $Y_1=Y_2$, so $Y\in\Q\cap O_K=\Z$; positivity gives $Y\in\Z^+$.

Dividing \eqref{eq:realW} by $Y^n$ yields
\begin{equation}\label{eq:real-series}
 \sum_{j=1}^n\frac{x_j}{Y^j}\in\Q.
\end{equation}
Suppose inductively that $x_1,\ldots,x_{m-1}\in\Z$. Subtracting their terms from \eqref{eq:real-series} and multiplying by $Y^m$ gives
\[
 x_m+\sum_{j=m+1}^n\frac{x_j}{Y^{j-m}}\in\Q.
\]
Also, 
\[
 Y=2c_m\prod_{j=m+1}^n(x_j^2+1),
 \qquad
 c_m=\prod_{j=1}^{m}(x_j^2+1),
\]
and both real embeddings of $c_m$ are at least $1$. By Lemma \ref{Lem3.2},  $x_m\in\Z$. Induction proves the theorem.
\end{proof}

Choose a nontrivial positive integer solution $a_0^2-db_0^2=1$. Replacing the corresponding unit by its square, put
\[
 A=a_0^2+db_0^2=2a_0^2-1,
 \qquad
 B=2a_0b_0,
 \qquad
 E=A^2-1=dB^2.
\]
Then $A$ is odd and $E$ is even. Define integer sequences $X_n,Y_n$ by
\begin{equation}\label{eq:denef-seq}
 X_n+Y_n\sqrt E=(A+\sqrt E)^n\qquad(n\ge0).
\end{equation}
In particular $X_n$ is odd for every $n$.

\begin{lemma}[Denef \cite{D75}]\label{lem:denef}

{\rm (i)} If $x,y\in O_K$ satisfy $x^2-Ey^2=1$, then $y^2\in\N$.

{\rm (ii)} For  $n,k\in\Z^+$, we have
\[
 Y_{nk}^2\equiv k^2Y_n^2\pmod{Y_n^4}.
\]
\end{lemma}

\begin{lemma}\label{lem:linear-pell}
For any positive integers $n,k$ with $2\nmid k$, we have
\[
 Y_{nk}\equiv kY_n\pmod{Y_n^3}.
\]
\end{lemma}

\begin{proof}
Expanding $(X_n+Y_n\sqrt E)^k$ and comparing coefficients of $\sqrt E$ gives
\[
 \frac{Y_{nk}}{Y_n}
 =\sum_{j=0}^{(k-1)/2}
 \binom{k}{2j+1}X_n^{k-2j-1}E^jY_n^{2j}.
\]
Modulo $Y_n^2$, only the $j=0$ term remains. Thus
\[
 \frac{Y_{nk}}{Y_n}\equiv kX_n^{k-1}\pmod{Y_n^2}.
\]
Since $k-1$ is even and $X_n^2=1+EY_n^2$, one has $X_n^{k-1}\equiv1\pmod{Y_n^2}$. Multiplication by $Y_n$ proves the assertion.
\end{proof}

\begin{lemma}\label{lem:pell-div}
For every $H\in\Z\sm\{0\}$, we have  $H\mid Y_n$ for some $n\in\Z^+$. Such $n$ can be chosen arbitrarily large.
\end{lemma}

\begin{proof} This is trivial for $H=\pm1$. 

Now assume that $|H|\gs2$. Then the matrix
\[
 M=\begin{pmatrix}A&E\\1&A\end{pmatrix}
\]
has determinant $1$, and satisfies
\[
 M^n=\begin{pmatrix}X_n&EY_n\\Y_n&X_n\end{pmatrix}.
\]
Its image in the finite group $\operatorname{GL}_2(\Z/|H|\Z)$ has finite order. Hence $M^r\equiv I\pmod{|H|}$ for some $r$, and $H\mid Y_r$. Every positive multiple of $r$ has the same divisibility property.
\end{proof}

Based on the above lemmas and Theorem \ref{Th3.1}, we can prove Theorem \ref{Th1.2} completely
in a way similar to the proof of Theorem \ref{Th1.1}. We omit the details.

\end{document}